\documentclass{amsart}
\usepackage{amsmath}
\usepackage{amsfonts}
\usepackage{amsthm}
\usepackage{graphicx}
\usepackage[all]{xy}
\usepackage{amssymb}
\usepackage{caption}
\usepackage{subcaption}

\newcommand{\bo}[1]{\textbf{#1}}
\newcommand{\dsum}[0]{\displaystyle\sum}

\newcommand{\R}[0]{\mathbb{R}}
\newcommand{\Z}[0]{\mathbb{Z}}
\newcommand{\N}[0]{\mathbb{N}}
\newcommand{\C}[0]{\mathbb{C}}

\newcommand{\dint}[0]{\displaystyle\int}
\newcommand{\dlim}[0]{\displaystyle\lim}

\newcommand{\angbra}[1]{\left\langle #1\right\rangle}
\newcommand{\inner}[2]{\angbra{#1,#2}}
\newcommand{\nbar}[1]{\overline{#1}}
\newcommand{\tr}[0]{\text{tr}}
\newcommand{\ph}[0]{\varphi}

\newcommand{\B}[0]{\mathcal{B}}

\newcommand{\mlim}[0]{\dlim_{m\to\infty}}
\newcommand{\rnx}[0]{\R^n_{\neq0}}
\newcommand{\tx}[0]{t^\xi}
\newcommand{\vx}[0]{v^\xi}

\newtheorem{thm}{Theorem}[section]
\newtheorem{lemma}[thm]{Lemma}
\newtheorem{prop}[thm]{Proposition}
\newtheorem{coro}[thm]{Corollary}

\newcommand{\inj}[0]{\ar@{^{(}->}}
\newcommand{\surj}[0]{\ar@{->>}}
\newcommand{\bij}[0]{\ar@{^{(}->>}}
\newcommand{\lbij}[0]{\ar@{_{(}->>}}
\newcommand{\linj}[0]{\ar@{_{(}->}}
\newcommand{\parr}[0]{\ar@{.>}}
\newcommand{\pinj}[0]{\ar@{^{(}.>}}
\newcommand{\psurj}[0]{\ar@{.>>}}
\newcommand{\pbij}[0]{\ar@{^{(}.>>}}
\newcommand{\lin}[0]{\ar@{-}}
\newcommand{\llin}[0]{\ar@{=}}

\makeatletter
\newtheorem*{rep@theorem}{\rep@title}
\newcommand{\newreptheorem}[2]{%
\newenvironment{rep#1}[1]{%
 \def\rep@title{#2 \ref{##1}}%
 \begin{rep@theorem}}%
 {\end{rep@theorem}}}
\makeatother

\newreptheorem{thm}{Theorem}

\begin{document}
\title{An Elementary Eigenvalue Criterion for One-Parameter Dilation Groups to Admit a Continuous Wavelet} 

\author{Netanel Friedenberg}
\address{Department of Mathematics, Washington University in St. Louis, MO 63130}
\email{nati.friedenberg@wustl.edu}

\author{Peter M. Luthy}
\address{Department of Mathematics, Washington University in St. Louis, MO 63130}
\email{luthy@math.wustl.edu}

\author{Guido L. Weiss}
\address{Department of Mathematics, Washington University in St. Louis, MO 63130}
\email{guido@math.wustl.edu}


\begin{abstract}
\noindent In this article, we present a simple criterion for checking whether a one-parameter matrix group of dilations admits a continuous wavelet.  This criterion involves only checking that the eigenvalues of the symmetric part of the matrix have the same sign.  In $\R^2$ this criterion gives a complete characterization of such matrix groups.
\end{abstract}
\maketitle

\section{Introduction}

\ \indent\! A (classical, orthonormal) wavelet is a function $\psi\in L^2(\R)$ such that the collection $\{\psi_{j,k}:=2^{-j/2}\psi(2^{-j}\cdot-k)|k\in\Z,j\in\Z\}$ is an orthonormal basis for $L^2(\R)$. So, if $\psi$ is a wavelet and $f\in L^2(\R)$ then we have
\begin{equation}\label{waveletreproduce}
f=\dsum_{j,k\in\Z}\inner{f}{\psi_{j,k}}\psi_{j,k},
\end{equation}
where the equality is in $L^2(\R)$. Much is known about wavelets, and there are many very nice constructions of wavelets, such as those that arise from multiresolution analyses (MRAs). A question that naturally arises is: how does one properly generalize wavelets to $\R^n$? The obvious generalization of requiring that $\{\psi_{j,k}:=2^{-nj/2}\psi(2^{-j}\cdot-k)|k\in\Z^n,j\in\Z\}$ be an orthonormal basis for $L^2(\R^n)$ does produce a theory of $n$-dimensional wavelets. 
Some aspects of this theory do not satisfactorily match the 1-dimensional case, though. 
For example, the construction of wavelets from MRAs in $L^2(\R^n)$ is much more complicated. If one uses the dyadic dilations $2^{j}$ for the MRA, then $2^n-1$ functions are needed to produce a spanning set, which is a consequence of the fact that the function $x\mapsto 2x$ on $\R^n$ has determinant $2^n$ (see~\cite{weiss}, Section 4). There are many other choices for dilations, and it is not clear which ones, if any, constitute the proper generalization of the classical wavelets --- see, for example, the discussion in the beginning of \cite{haar2007}. 
As such, many other systems have been developed, such as composite dilation wavelets and shearlets (see~\cite{LWW2013} pp 69-72 for an overview of these and other systems).

There is another variant of the wavelets that generalizes quite nicely to $\R^n$. These systems use a continuous translation parameter, and so the wavelets that arise from them are called continuous wavelets. For any closed subgroup $G$ of $GL_n(\R)$ with a left Haar measure $\mu$, we consider the group $G_\#$ of all pairs $(a,b)\in G\times \R^n$ with multiplication given by $(\alpha,\beta)\cdot(a,b)=(\alpha a,b+a^{-1}\beta)$. This operation arises from the action $((a,b),x)\mapsto a(x+b)$ on $\R^n$. We therefore say that $G$ is the group of dilations, and $\R^n$ is the group of translations. Thus this group is related to the affine group $G\ltimes\R^n$ which is associated with the action $((a,b),x)\mapsto ax+b$. Indeed, the map $\phi:G_\#\to G\ltimes\R^n$ given by $\phi(a,b)=(a^{-1},b)$ is an anti-isomorphism of groups. Note that the element of a left Haar measure for $G_\#$ is given by $d\lambda(a,b)=d\mu(a)db$. We consider the unitary representation of $G_\#$ on $L^2(\R^n)$ given by 
$$\psi,(a,b)\mapsto \psi_{(a,b)}(x):=\frac{\psi((a,b)^{-1}x)}{|\det(a)|^{1/2}}=|\det(a)|^{-\frac{1}{2}}\psi(a^{-1}x-b).$$
$\psi\in L^2(\R^n)$ is called a continuous wavelet for $G$ if the reproducing formula
$$f(x)=\int_{G_\#}\inner{f}{\psi_{(a,b)}}\psi_{(a,b)}(x)d\lambda(a,b),$$
which is clearly the (direct) analog of (\ref{waveletreproduce}), is true for all $f\in L^2(\R^n)$. 
As shown in~\cite{weiss}, $\psi\in L^2(\R^n)$ is such a wavelet if and only if for almost every $\xi\in\rnx:=\R^n-\{0\}$,
\begin{equation}\label{admisscond}
\Delta_\psi(\xi):=\dint_G |\hat{\psi}(a^T\xi)|^2 d\mu(a)=1
\end{equation}
where $a^T$ denotes the transpose of the matrix $a$. This formula is a natural analog of the dyadic Calder$\acute{\text{o}}$n condition for classical wavelets on $L^2(\R)$, that $\dsum_{j\in\Z}|\hat{\psi}(2^j\xi)|=1$ for a.e.\ $\xi\in\R$, which is also the specific instance of (\ref{admisscond}) when $G$ the dyadic dilation group $\{2^j|j\in\Z\}$. In the case of classical wavelets, where $\Z$ translations are used, this equation is not sufficient to show that $\psi$ is a wavelet. Rather, $\psi$ must also satisfy $t_q(\xi):=\dsum_{j\geq0}\hat{\psi}(2^j\xi)\nbar{\hat{\psi}(2^j(\xi+q))}=0$ almost everywhere, for every odd integer $q$, which stems from the fact that dyadic dilations and integer translations do not form a group together. Specifically, letting $D^j\psi=2^{-j/2}\psi(2^{-j}\cdot)$ and $T_k\psi=\psi(\cdot-k)$, we have that $D^jT_k=T_{2^jk}D^j$, but $2^jk$ may not be an integer even when $j$ and $k$ are. That the single condition is sufficient when $\R$ translations are used indicates that $\R$ translation is in some ways nicer. Indeed, we immediately see that any classical wavelet is a continuous wavelet for the dyadic dilation group $\{2^j|j\in\Z\}$, and it is not hard to show that (up to renormalization) it is also a continuous wavelet for $\R-\{0\}$.

Not every closed subgroup $G$ of $GL_n(\R)$ allows a $\psi$ satisfying (\ref{admisscond}). If $G$ does admit a continuous wavelet, then $G$ is called admissible.  The question of admissibility or non-admissibility of matrix groups has been studied previously, and an (almost) characterization was produced in \cite{LWWW2002}; the discussion therein involves the study of so-called $\epsilon$-stabilizers which are not so easy to understand.  At present, we are interested in determining which 1-parameter groups are admissible. These are groups of the form $\{e^{tM}:t\in\R\}$ where $M$ is a fixed $n\times n$ matrix.  We will see that there is an elementary criterion for checking when such matrix groups are admissible. In section 2 we will show that if the eigenvalues of the symmetric part of $M$ are all nonzero reals that have the same sign, then $\{e^{tM}:t\in\R\}$ is admissible. This will then be used in section 3 to give a simple characterization of the admissible one-parameter groups over $\R^2$.

Our idea was to show that if $D$ was a diagonal matrix with positive eigenvalues and $A$ was an antisymmetric matrix, then $\{e^{t(D+A)}:t\in\R\}$ was an admissible group. Recall that every matrix $X$ can be written as the sum of a symmetric matrix and an antisymmetric matrix, $X=\frac{A+A^T}{2}+\frac{A-A^T}{2}$, so this would be a good start on the problem. We would show this by noting that for positive $t$, $e^{tD}$ is a matrix that increases the norm of a vector, and for any $t\in\R$, $e^{tA}$ is an orthogonal matrix. The problem with this approach is that if $A$ and $D$ do not commute, then $e^{t(D+A)}$ may not be equal to $e^{tD}e^{tA}$. The Lie product formula, however, enables us to prove this and more general results in spite of the matrix exponential not splitting in this way. Before proving this formula, we first review a few basic facts about the matrix exponential and logarithm.

Recall that the matrix exponential $\exp: M_n(\C)\to GL_n(\C)$ mapping $X\mapsto e^{X}$ is a smooth function such that\\
\indent1. $e^0=I$,\\
\indent2. $e^{X^T}=(e^X)^T$,\\
\indent3. if $X$ and $Y$ are commuting matrices, then $e^{X+Y}=e^Xe^Y$,\\
\indent4. for any invertible matrix $C$, $e^{CXC^{-1}}=Ce^XC^{-1}$,\\
\indent5. $\|e^X\|\leq e^{\|X\|}$, and\\
\indent6. $e^X=I+X+O(\|X\|^2)$.\\
Also recall that the matrix logarithm $\log(A)=\dsum_{m=1}^\infty(-1)^{m+1}\frac{(A-I)^m}{m}$, when restricted to the domain $\{A\in M_n(\C):\|A-I\|<1\}$, is a well-defined matrix-valued function such that\\
\indent1. $e^{\log(A)}=A$ if $\|A-I\|<1$, and\\
\indent2. $\log(I+B)=B+O(\|B\|^2)$.\\

\begin{thm}[\bo{Lie Product Formula}]\label{lieprod}
Let $X$ and $Y$ be any $n\times n$ complex matrices. Then
$$e^{X+Y}=\mlim\left(e^{\frac{X}{m}}e^{\frac{Y}{m}}\right)^m.$$
\end{thm}\begin{proof}
Fix $X$ and $Y$. We have $e^\frac{X}{m}e^\frac{Y}{m}=I+\frac{X}{m}+\frac{Y}{m}+O(m^{-2})$. Since as $m\to\infty$ this goes to $I$, $e^\frac{X}{m}e^\frac{Y}{m}$ is eventually in the domain of the logarithm. So we can write
$$\log\left(e^\frac{X}{m}e^\frac{Y}{m}\right)=\log\left(I+\frac{X}{m}+\frac{Y}{m}+O(m^{-2})\right)=\frac{X}{m}+\frac{Y}{m}+O(m^{-2}).$$
Exponentiating the logarithm now gives us 
$e^\frac{X}{m}e^\frac{Y}{m}=\exp\left(\frac{X}{m}+\frac{Y}{m}+O(m^{-2})\right),$
and so 
$\left(e^\frac{X}{m}e^\frac{Y}{m}\right)^m=\exp(X+Y+O(m^{-1}))$. 
So, by the continuity of the exponential, we have 
\begin{align*}
\mlim\left(e^\frac{X}{m}e^\frac{Y}{m}\right)^m&=\mlim\exp(X+Y+(m^{-1}))\\
&=\exp\left(\mlim(X+Y+O(m^{-1}))\right)=\exp(X+Y),
\end{align*}
which is the Lie product formula.
\end{proof}

\section{Criteria for Admissibility of Groups in Terms of Eigenvalues}

\ \indent\! The main goal of this section is to prove
\begin{thm}\label{thm:mainthm}
Let $M$ be any $n\times n$ real matrix that is orthogonally diagonalizable, and suppose that all of the eigenvalues of $M$ are nonzero and have the same sign. Let $A$ be any anti-symmetric $n\times n$ matrix. Then $G_{M+A}=\{Q_t=e^{t(M+A)}:t\in\R\}$ is admissible.
\end{thm}
First, however, we need a technical topological result. 
Note that from now on all matrices will be assumed to be real unless stated otherwise.

\begin{lemma}\label{lemma:comphaushomeo}
Let $X$ be a compact space, $Y$ be a Hausdorff space, and $f:X\to Y$ a continuous function. Then $f$ is closed. In particular, if $f$ is a continuous bijection, then $f$ is a homeomorphism.
\end{lemma}\begin{proof} Elementary topology.\end{proof}

\begin{prop}\label{prop:bighomeo}
Let $X$ be a Hausdorff space, $Y$ be any topological space, and $g:X\to Y$ an open bijection. Suppose that there exists a collection $\B$ of compact subsets of $Y$ with the property that for any $x\in X$ there is some $B\in\B$ such that $g^{-1}[B]$ contains an open neighborhood of $x$. Then $g$ is a homeomorphism.
\end{prop}
Note: The case in which we will use this proposition is with $X=Y=\R^n$. $\B$ will be the collection of closed balls centered at the origin, and $g$ will be a bijection from $\R^n$ to itself such that $g^{-1}$ is continuous and the images under $g^{-1}$ of sufficiently large closed balls contain large open balls centered at the origin. The proposition is only stated in such generality because the general result is no more difficult to prove than the specific case.
\begin{proof}
Let $f=g^{-1}$, so $f$ is continuous. For any $B\in\B$, $f|_B:B\to f[B]$ is a continuous bijection from the compact space $B$ to the Hausdorff space $f[B]\subset X$. So by Lemma \ref{lemma:comphaushomeo} $f|_B:B\to f[B]$ is a homeomorphism, so its inverse, $(f|_B)^{-1}=g|_{f[B]}$, is continuous. Now, for any $x\in X$ find some $B_x\in\B$ such that $f[B_x]=g^{-1}[B_x]$ contains an open neighborhood of $x$. Since $g|_{f[B_x]}:f[B_x]\to B_x$ is continuous at $x$, and since $f[B_x]$ contains an open neighborhood of $x$ in $X$, this gives us that $g$ is continuous at $x$. Since this is true for all $x\in X$, $g$ is continuous on all of $X$. So $g:X\to Y$, being a continuous open bijection, is a homeomorphism.
\end{proof}

We are now properly equipped to begin proving our main theorem, which we restate here.

\begin{repthm}{thm:mainthm}
Let $M$ be any $n\times n$ real matrix that is orthogonally diagonalizable, and suppose that all of the eigenvalues of $M$ are nonzero and have the same sign. Let $A$ be any anti-symmetric $n\times n$ matrix. Then $G_{M+A}=\{Q_t=e^{t(M+A)}:t\in\R\}$ is admissible.
\end{repthm}\begin{proof}
Letting $Q_{S}=\{e^{t(M+A)}:t\in S\}$ for all $S\subset\R$, we see that $e^{t_0(M+A)}Q_{[a,b]}=Q_{[a+t_0,b+t_0]}$ for any $t_0\in\R$. Note that this is the same as $\R$ acting on its closed intervals by addition, and the sets $Q_{[a,b]}$ generate all of the borel sets of $G_{M+A}$. So, choosing the (left) Haar measure $\mu$ on $G_{M+A}$ such that $\mu(Q_{[0,1]})=1$, we have that $\mu(Q_S)$ is just the Lebesgue measure of $S$ for any measurable $S\subset\R$. Hence, integrating over $G_{M+A}$ with the measure $\mu$ gives the same result as integrating over $\R$ with the Lebesgue measure, considering any instance of the group element $e^{t(M+A)}$ as a function of $t\in\R$.

We claim that we may assume, without loss of generality, that the eigenvalues of $M$ are all positive. To see this, note that if $M$ has all negative eigenvalues, reparameterizing by $t\mapsto-t$ we get $G_{M+A}=\{e^{t(-M-A)}:t\in\R\}$. Since $-M$ will still be orthogonally diagonalizable, but have as its eigenvalues the eigenvalues of $M$ times $-1$, and since $-A$ is also anti-symmetric, we can view the group this way and have all of the eigenvalues of $M$ positive.

Now, we can find an orthogonal matrix $O$ such that $M=ODO^T$ where
$$D=\begin{bmatrix}
\lambda_1&&0\\
&\ddots&\\
0&&\lambda_n\end{bmatrix}$$
with $\lambda_1\geq\cdots\geq\lambda_n>0$. For any $t\in\R$ we have $e^{tM}=Oe^{tD}O^T$ and 
$$e^{tD}=\begin{bmatrix}
e^{t\lambda_1}&&0\\
&\ddots&\\
0&&e^{t\lambda_n}\end{bmatrix}.$$
Note that for $t>0$ we have that $e^{t\lambda_1}\geq\cdots\geq e^{t\lambda_n}>1$, so for any $v\in\R^n$, we have that $e^{t\lambda_1}\|v\|\geq\|e^{tD}v\|\geq e^{t\lambda_n}\|v\|$. Similarly, for $t<0$ and any $v\in\R^n$ we have that $e^{t\lambda_1}\|v\|\leq\|e^{tD}v\|\leq e^{t\lambda_n}\|v\|$. Also, for any $t\in\R$, we have $(e^{tA})^T=e^{tA^T}=e^{-tA}=(e^{tA})^{-1}$, so $e^{tA}$ is orthogonal.

We claim that for $t>0$ and any $v\in\R^n$, $\left\|e^{t(M+A)}v\right\|\geq e^{t\lambda_n}\|v\|$. Towards this, note that for any $m\in\N$, because multiplication by an orthogonal matrix is an isometry, we have 
\begin{align*}
\left\|e^\frac{tM}{m}e^\frac{tA}{m}v\right\|&=\left\|Oe^\frac{tD}{m}O^Te^\frac{tA}{m}v\right\|
=\left\|e^\frac{tD}{m}O^Te^\frac{tA}{m}v\right\|\\
&\geq e^{\frac{t}{m}\lambda_n}\left\|O^Te^\frac{tA}{m}v\right\|
= e^{\frac{t}{m}\lambda_n}\left\|v\right\|.
\end{align*}
Iterating this (by plugging in $e^\frac{tM}{m}e^\frac{tA}{m}v$ for $v$) $m$ times, we get that $\left\|\left(e^\frac{tM}{m}e^\frac{tA}{m}\right)^m v\right\|\geq\left(e^{\frac{t}{m}\lambda_n}\right)^m \|v\|=e^{t\lambda_n} \|v\|$. So now, by the Lie product formula, we have 
\begin{align*}\left\|e^{t(M+A)}v\right\|
&=\left\|\mlim \left(e^\frac{tM}{m}e^\frac{tA}{m}\right)^m v\right\|
=\mlim \left\|\left(e^\frac{tM}{m}e^\frac{tA}{m}\right)^m v\right\|\\
&\geq\mlim e^{t\lambda_n} \|v\|=e^{t\lambda_n}\|v\|,\end{align*}
as we claimed. Note that the exact same method gives us that for $t>0$ and $v\in\R^n$, $e^{t\lambda_1}\|v\|\geq\|e^{t(M+A)}v\|$, and that for $t<0$ and $v\in\R^n$, $e^{t\lambda_1}\|v\|\leq\left\|e^{t(M+A)}v\right\|\leq e^{t\lambda_n}\|v\|$.

This immediately shows us that for a fixed $v\in\rnx:=\R^n-\{0\}$, the function $\left\|e^{t(M+A)}v\right\|$ is strictly increasing in $t$. For if $t_0<t_1$, then we have $\left\|e^{t_1(M+A)}v\right\|=\left\|e^{(t_1-t_0)(M+A)}e^{t_0(M+A)}v\right\|\geq e^{(t_1-t_0)\lambda_n}\left\|e^{t_0(M+A)}v\right\|>\left\|e^{t_0(M+A)}v\right\|$. Moreover, we now know the limiting behavior of $\left\|e^{t(M+A)}v\right\|$. As $t\to\infty$, $e^{t\lambda_n}\to\infty$, and so $\left\|e^{t(M+A)}v\right\|\geq e^{t\lambda_n}\|v\|$ must also go to $\infty$ as well. Similarly, as $t\to-\infty$, $e^{t\lambda_n}\to0$, and so $0\leq\left\|e^{t(M+A)}v\right\|\leq e^{t\lambda_n}\|v\|$ gives us that $\left\|e^{t(M+A)}v\right\|\to 0$ as well. So we have shown that $t\mapsto\left\|e^{t(M+A)}v\right\|$, for any fixed $v\in\rnx$, is a strictly increasing function, the range of which is all of $\R^+:=\{r\in\R:r>0\}$. 
This tells us that for a fixed $v\in\rnx$ and a given positive real number $r$, there is a unique $e^{t(M+A)}\in G_{M+A}$ such that $\left\|e^{t(M+A)}v\right\|=r$. That is, the orbit of $v$ under the action of $G_{M+A}$ has a unique representative (in $\rnx$) with norm $r$. In particular, this means that when considering the orbits of $\rnx$ under $G_{M+A}$ we are just considering the orbits of those elements in $S^{n-1}$, the unit sphere in $\R^n$. So for any $\xi\in\rnx$, we can write it uniquely as $\xi=e^{t(M+A)}v$ with $t\in\R$ and $v\in S^{n-1}$.

At this point we need the following lemma, which is rather technical.

\begin{lemma}
Let $g$ be the function mapping any $\xi\in\rnx$ to the unique $t\in\R$ such that $\xi=e^{t(M+A)}v$ for some $v\in S^{n-1}$. Then $g:\rnx\to\R$ is continuous.
\end{lemma}\begin{proof}
Given a $\xi\in\rnx$, we have unique $t^\xi\in\R$ and $v^\xi\in S^{n-1}$ such that $\xi=e^{t^\xi(M+A)}v^\xi$. Let $q:\rnx\to S^{n-1}$ be the function $\xi\mapsto v^\xi$, and note that the mapping $\xi\mapsto t^\xi$ is $g:\rnx\to\R$. Now, let $s:\R\times S^{n-1}\to\rnx$ be defined by $s(t,v)=e^tv$, and let $G:\rnx\to\rnx$ be the composition $G(\xi)=s(g(\xi),q(\xi))=e^{t^\xi}v^\xi=:u^\xi$. While $\xi$ ranges over $\rnx$, $g(\xi)$ ranges over all of $\R$ and $q(\xi)$ ranges over all of $S^{n-1}$, and they do so independently of each other (i.e.\ $(t^\xi,v^\xi)$ ranges over all of $\R\times S^{n-1}$). This means that $e^{g(\xi)}$ ranges over all of $\R^+$, and so $G(\xi)$ ranges over all of $\rnx$. Noting that each $\xi\in\rnx$ gives rise to a unique pair $(g(\xi),q(\xi))$ (because $t^\xi$ and $v^\xi$ determine $\xi=e^{t^\xi(M+A)}v^\xi$), and that $s$ is injective, we see that $G$ is also 1-1. So $G$ is a bijection from $\rnx$ to itself. We produce the following diagrams:


\renewcommand{\figurename}{Diagram}

\begin{figure}[h!]
\centering
\begin{subfigure}[b]{0.4\textwidth}
$$\xymatrix@R=7pt {
&\R\ar `d[dr] [dr]^-s&\\
\rnx\ar@/^3pc/[rr]^{G}\ar[ur] |-{g}\ar[dr] |-{q}&&\rnx\\
&S^{n-1}\ar `u[ur] [ur]&}$$
\caption{Mappings between spaces}
\end{subfigure}
\quad
\begin{subfigure}[b]{0.4\textwidth}
$$\xymatrix@R=7pt {
&\;t^\xi\;\ar `d[dr] [dr]^-s&\\
\;\xi\;\ar@/^3pc/[rr]^{G}\ar[ur] |-{g}\ar[dr]|-{q}&&\;u\;\\
&\;v^\xi\;\ar `u[ur] [ur]&}$$
\caption{Corresponding variable names}
\end{subfigure}
\caption{The maps $g$, $q$, $s$, and the composition $G$.}
\end{figure}
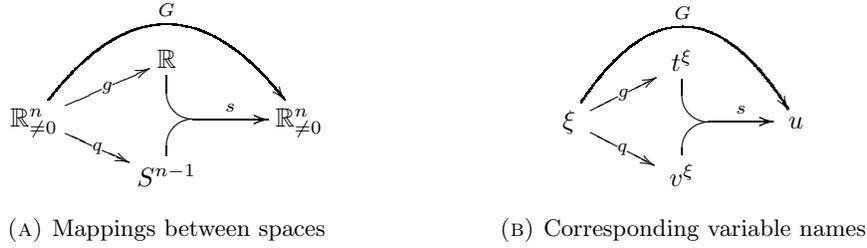


Inverting this process, define $f:\rnx\to\R$ by $f(u)=\log(\|u\|)=:t_u$ and $p:\rnx\to S^{n-1}$ by $p(u)=\dfrac{u}{\|u\|}=:v_u$. Since $\|u\|\neq0$ on $\rnx$, $f$ and $p$ are continuous. Clearly $t\mapsto e^{t(M+A)}$ is a continuous function, and so we have that $r:\R\times S^{n-1}\to\rnx$ defined by $r(t,v)=e^{t(M+A)}v$ is also continuous. So now, letting $F:\rnx\to\rnx$ be the composition $F(u)=r(f(u),p(u))=e^{t_u(M+A)}v_u=:\xi_u$, $F$ is continuous as well. The corresponding diagrams are:

$$
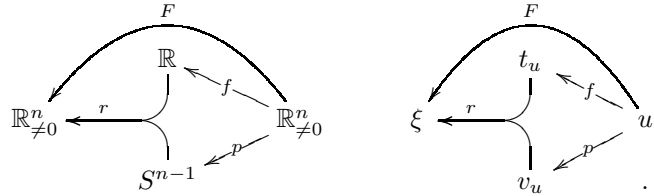
\xymatrix@R=7pt {
&\R\ar `d[dl] [dl]_-r&
&
&\;t_u\;\ar `d[dl] [dl]_-r&\\
\rnx&&\rnx\ar[ul] |-{f}\ar[dl] |-{p}\ar@/_3pc/[ll]_{F}
&
\;\xi\;&&\;u\;\ar[ul] |-{f}\ar[dl] |-{p}\ar@/_3pc/[ll]_{F}\\
&S^{n-1}\ar `u[ul] [ul]&
&
&\;v_u\;\ar `u[ul] [ul]&.}$$


Also, it is quite obvious from our definitions that the full diagrams,

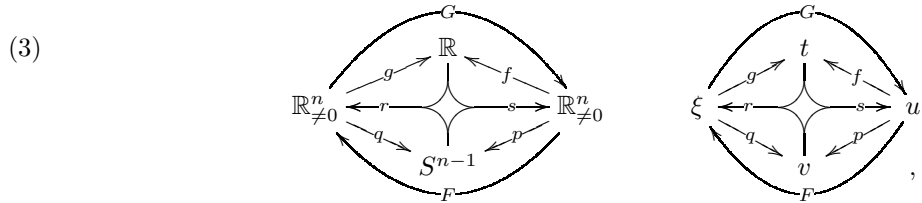
\begin{equation}\label{fulldiagram}
\xymatrix@R=7pt {
&\R 
	\ar `d[dr] [dr] |-{s}
	\ar `d[dl] [dl] |-{r}&
&
&\;t\;
	\ar `d[dr] [dr] |-{s}
	\ar `d[dl] [dl] |-{r}&\\
\rnx 
	\ar@/^3pc/[rr] |-{G}
	\ar[ur] |-{g}\ar[dr] |-{q}
&&\rnx 
	\ar[ul] |-{f}\ar[dl] |-{p}
	\ar@/^2.75pc/[ll] |-{F}
&
\;\xi\;
	\ar@/^3pc/[rr] |-{G}
	\ar[ur] |-{g}\ar[dr] |-{q}
&&\;u\;
	\ar[ul] |-{f}\ar[dl] |-{p}
	\ar@/^2.75pc/[ll] |-{F}\\
&S^{n-1} 
	\ar `u[ur] [ur] |-{\hole}
	\ar `u[ul] [ul] |-{\hole}&
&
&\;v\;
	\ar `u[ur] [ur] |-{\hole}
	\ar `u[ul] [ul] |-{\hole}&,}
\end{equation}


\noindent commute. Specifically, it is clear that $F=G^{-1}$.

Note that as $u\to0$, $t_u=\log(\|u\|)\to-\infty$. Now we know that for $t_u\leq 0$ (i.e.\ $\|u\|\leq1$) $\|\xi_u\|=\|e^{t_u(M+A)}v_u\|\leq e^{t_u\lambda_n}\|v_u\|=e^{t_u\lambda_n}$, and so as $u\to 0$ and so $t_u\to-\infty$, we must have that $F(u)=\xi_u$ goes to 0. So $F_0:\R^n\to\R^n$ defined by $F_0(u)=\begin{cases}F(u)&\text{if }u\neq0\\0&\text{if }u=0\end{cases}$ is also a continuous bijection. Note that $F_0$ is the inverse of $G_0:\R^n\to\R^n$ given by $G_0(\xi)=\begin{cases}G(\xi)&\text{if }\xi\neq0\\0&\text{if }\xi=0\end{cases}$.

For any $\rho\in\R^+$, let $B(\rho)=\{y\in\R^n:\|y\|<\rho\}$ and $\nbar{B}(\rho)=\{y\in\R^n:\|y\|\leq \rho\}$ be the open and closed balls, respectively, that have radius $\rho$ and are centered at $0$ in $\R^n$. Recall that the $\nbar{B}(\rho)$ are compact.

Now fix $\rho\geq1$. If $u\notin\nbar{B}(\rho)$, then $\|u\|>\rho\geq1$, so $t_u=\log(\|u\|)>\log(\rho)\geq0$. So,
$$\|F_0(u)\|=\|F(u)\|=\|e^{t_u(M+A)}v_u\|\geq e^{t_u\lambda_n}\|v_u\|>e^{\log(\rho)\lambda_n}=\rho^{\lambda_n}.$$
This shows that if $u\in\R^n-\nbar{B}(\rho)$, then $F_0(u)\in\R^n-\nbar{B}(\rho^{\lambda_n})$, i.e.\ $F_0\left(\R^n-\nbar{B}(\rho)\right)\subset\R^n-\nbar{B}\left(\rho^{\lambda_n}\right)$. Since $F_0$ is a bijection we get
$$F_0\left(\nbar{B}(\rho)\right)=\R^n-F_0\left(\R^n-\nbar{B}(\rho)\right)\supset\nbar{B}\left(\rho^{\lambda_n}\right)\supset B\left(\rho^{\lambda_n}\right).$$

So, given $x\in\R^n$, let $\rho_x=1+\|x\|^{1/\lambda_n}$. Since $\rho_x\geq1$, $B\left(\rho_x^{\lambda_n}\right)\subset F_0\left(\nbar{B}(\rho_x)\right)$. But we also have that $\|x\|=\left(\|x\|^{1/\lambda_n}\right)^{\lambda_n}<\rho_x^{\lambda_n}$ where the last inequality holds because $\rho_x>\|x\|^{1/\lambda_n}$ and $\lambda_n>0$. So $x\in B\left(\rho_x^{\lambda_n}\right)\subset F_0\left(\nbar{B}(\rho_x)\right)$, and since $B\left(\rho_x^{\lambda_n}\right)$ is an open set in $\R^n$, $F_0\left(\nbar{B}(\rho_x)\right)$ contains an open neighborhood of $x$ in $\R^n$.

Note that since $F_0$ is a continuous bijection, $G_0=F_0^{-1}:\R^n\to\R^n$ is an open bijection. We also have that $\B=\{\nbar{B}(\rho):\rho\geq1\}$ is a collection of compact sets with the property that for any $x\in\R^n$, there is some $\nbar{B}(\rho_x)\in\B$ such that $G_0^{-1}\left[\nbar{B}(\rho_x)\right]=F_0\left[\nbar{B}(\rho_x)\right]$ contains an open neighborhood of $x$ in $\R^n$. Since $\R^n$ is Hausdorff, Proposition \ref{prop:bighomeo} thus tells us that $G_0$ is a homeomorphism. Specifically, $G_0$ is continuous, and so $G=G_0|_{\rnx}$ is also continuous. Looking at (\ref{fulldiagram}) we see that $g=f\circ G$, and since we now know that $f$ and $G$ are both continuous, $g$ is continuous as well.
\end{proof}

Returning to the proof of our theorem, fix some $\ph\in L^2(\R)$ with $\|\ph\|_2=1$ that is continuous and is supported in $(-\infty,N]$ for some $N\in\R^+$, such that $\dlim_{t\to-\infty}\ph(t)=0$ (here $\|\cdot\|_2$ denotes the $L^2$ norm). Now define $\hat{\psi}:\R^n\to\C$ by $\hat{\psi}(\xi)=\begin{cases}
\ph(t^\xi)&\text{if }\xi\neq0\\
0&\text{if }\xi=0\end{cases}$. Clearly $\hat{\psi}$ is continuous at any $\xi\neq0$, because it is a composition of continuous functions on $\rnx$. Note that as $\xi\to0$, we have $e^{t^\xi\lambda_1}\leq\left\|e^{t^\xi(M+A)}v^\xi\right\|=\|\xi\|\to 0$, so $t^\xi\to-\infty$, and so $\dlim_{\xi\to0}\hat{\psi}(\xi)=\dlim_{\xi\to0}\ph(t^\xi)=\dlim_{t\to-\infty}\ph(t)=0=\hat{\psi}(0)$, so $\hat{\psi}$ is continuous on all of $\R^n$.

We claim that $\hat{\psi}$ is compactly supported. Note that if $\tx<0$, then $\|\xi\|=\left\|e^{\tx(M+A)}\vx\right\|\leq e^{\tx\lambda_n}<1$. Also if $\tx=0$, then $\|\xi\|=\left\|e^{t^\xi(M+A)}v^\xi\right\|=1$. So if $\tx\leq 0$, then $\xi\in\nbar{B}(1)$. On the other hand, we can consider $\xi$ with $\tx>0$ such that $\hat{\psi}(\xi)\neq0$. For such $\xi$ we have $0\neq\hat{\psi}(\xi)=\ph(\tx)$, so $\tx\leq N$, and so $\|\xi\|=\left\|e^{\tx(M+A)}\vx\right\|\leq e^{\tx\lambda_1}\leq e^{N\lambda_1}$. So letting $R=\max\{1,e^{N\lambda_1}\}$, we have that $\hat{\psi}$ is supported in $\nbar{B}(R)$.

So now, since $\hat{\psi}$ is a compactly supported continuous function on $\R^n$, we have that $\hat{\psi}\in L^2(\R^n)$, and so $\hat{\psi}$ is actually the Fourier transform of some $\psi\in L^2(\R^n)$.

Noting that $M^T$ and $A^T$ satisfy the same criteria we placed on $M$ and $A$, we see that all of our constructions and proofs work for the group $G_{M^T+A^T}$ as well. So let $\psi_T$ be the $\psi$ that we get for the group $G_{M^T+A^T}$. We claim that $\psi_T$ is a wavelet for $G_{M+A}$. For any nonzero $\xi\in\R^n$, let $t_0$ and $v_0$ be the unique elements of $\R$ and $S^{n-1}$, respectively, such that $\xi=e^{t_0(M^T+A^T)}v_0$. Then we have

\begin{align*}
\dint_{G_{M+A}} |\hat{\psi}_T(Q^T\xi)|^2 d\mu(Q)&=\dint_\R \left|\hat{\psi}_T\left(\left(e^{t(M+A)}\right)^T e^{t_0(M^T+A^T)}v_0\right)\right|^2 dt\\
&=\dint_\R \left|\hat{\psi}_T\left(e^{(t+t_0)(M^T+A^T)}v_0\right)\right|^2 dt\\
&=\dint_\R \left|\hat{\psi}_T\left(e^{t(M^T+A^T)}v_0\right)\right|^2 dt\\
&=\dint_\R |\ph(t)|^2 dt=\|\ph\|_2\\
&=1,
\end{align*}
i.e.\ $\psi_T$ satisfies the admissibility condition (\ref{admisscond}) for $G_{M+A}$. So $G_{M+A}$ is admissible.
\end{proof}

\begin{coro}\label{coro:symmpart}
Let $X$ be any $n\times n$ real matrix. If the eigenvalues of the symmetric part of $X$, $\dfrac{X+X^T}{2}$, are all nonzero reals that have the same sign, then $G_X=\{e^{tX}:t\in\R\}$ is admissible.
\end{coro}\begin{proof}
Let $M_X=\frac{X+X^T}{2}$, the symmetric part of $X$, and let $A_X=\frac{X-X^T}{2}$, the anti-symmetric part of $X$. Then the spectral theorem tells us that $M_X$ is orthogonally diagonalizable. Since we already know that all of the eigenvalues of $M_X$ are nonzero and have the same sign, and that $A_X$ is anti-symmetric, we can apply Theorem \ref{thm:mainthm}, so $G_X$ is admissible.
\end{proof}

Note: If $M$ is as in Theorem \ref{thm:mainthm}, we can write $M=ODO^T$ with $O$ orthogonal and $D$ diagonal. Then $M^T=(O^T)^TD^TO^T=ODO^T=M$, so $M$ is symmetric. Thus, Corollary \ref{coro:symmpart} still carries the full strength of Theorem \ref{thm:mainthm}.

\begin{lemma}\label{lemma:simadmiss}
Let $X$ and $Y$ be real $n\times n$ matrices that are similar in $M_n(\R)$. If $G_Y=\{e^{tY}:t\in\R\}$ is admissible, then $G_X$ is admissible as well.
\end{lemma}\begin{proof}
Write $X=SYS^{-1}$ with $S\in GL_n(\R)$. Letting $\phi$ be a wavelet for $G_Y$ we have 
\begin{equation}\label{Yadmiss}
\Delta_{\phi,Y}(\xi)=\dint_{G_Y} |\hat{\phi}((e^{tY})^T\xi)|^2 d\mu(e^{tY})=1\text{\quad for a.e. }\xi\in\rnx.
\end{equation}
Letting $P=S^T$ and defining $\psi\in L^2(\R^n)$ by $\hat\psi(\xi)=\hat\phi(P^{-1}\xi)$, we have that 
$$\hat\psi((e^{tX})^T\xi)=\hat\phi\left((S^{-1})^T(e^{tSYS^{-1}})^T\xi\right)=\hat\phi((e^{tY})^T P^{-1}\xi),$$
and so (choosing the Haar measures as in the proof of Theorem \ref{thm:mainthm})
$$
\Delta_{\psi,X}(\xi)=\dint_\R\left|\hat{\psi}\left(\left(e^{tX}\right)^T\xi\right)\right|^2 dt
=\dint_\R\left|\hat{\phi}\left(\left(e^{tY}\right)^T P^{-1}\xi\right)\right|^2 dt
=\Delta_{\phi,Y}(P^{-1}\xi).
$$
Since $\xi\mapsto P^{-1}\xi$ preserves sets of measure zero, this last equation and (\ref{Yadmiss}) give us that $\psi$ is a wavelet for $G_X$, which is thus admissible.
\end{proof}

\begin{coro}\label{coro:complexdiagadmiss}
Let $X$ be any $n\times n$ real matrix. If $X$ is diagonalizable over $\C$ and the real parts of the eigenvalues of $X$ are all nonzero and have the same sign, then $G_X=\{e^{tX}:t\in\R\}$ is admissible.
\end{coro}\begin{proof}
Note that if $X$ has a complex eigenvalue $\lambda$, since $X$ is real it also has $\nbar\lambda$ as an eigenvalue. So we can write the list of eigenvalues of $X$ (including multiplicities), as $\lambda_1,\nbar\lambda_1,\cdots,\lambda_k,\nbar\lambda_k,\mu_{2k+1},\cdots,\mu_n$ with the $\lambda_j=a_j+ib_j$ ($b_j\neq 0$) and the $\mu_l$ real. So using conjugation by matrices in $GL_n(\C)$, $X$ is similar to the diagonal matrix $D$ with this list of entries down its diagonal. But letting $B=\begin{bmatrix}1&1\\i&-i\end{bmatrix}$, we have $B^{-1}=\frac{1}{2}\begin{bmatrix}1&-i\\1&i\end{bmatrix}$ and $B\begin{bmatrix}\lambda_j&0\\0&\nbar\lambda_j\end{bmatrix}B^{-1}=\begin{bmatrix}a_j&b_j\\-b_j&a_j\end{bmatrix}$. So letting $C=\underbrace{B\oplus\cdots\oplus B}_{k}\oplus I_{n-2k}$, we have that $X$ is similar, over $\C$, to 
$$A:=CDC^{-1}=\begin{bmatrix}
a_1&b_1\\
-b_1&a_1\\
&&\ddots\\
&&&a_k&b_k\\
&&&-b_k&a_k\\
&&&&&\mu_{2k+1}\\
&&&&&&\ddots\\
&&&&&&&\mu_n
\end{bmatrix}.$$
But two real matrices that are similar in $M_n(\C)$ are also similar in $M_n(\R)$ (because they have the same rational canonical form). Now it is clear that the symmetric part of $A$ is the diagonal matrix with entries (and so eigenvalues) $a_1,a_1,\ldots,a_k,a_k,\mu_{2k+1},\ldots,\mu_n$, which by hypothesis are all nonzero and have the same sign. So by Corollary \ref{coro:symmpart} $G_A$ is admissible, and so the previous lemma gives us that $G_X$ is admissible.
\end{proof}

\section{Further Results on One-Parameter Groups}

\ \indent\! The main theorem and its corollaries provide useful and easy to apply tools for proving that one-parameter groups are admissible. On the other hand, they do not give us any way to show that a one-parameter group is not admissible. We start this section by showing exactly which diagonal matrices generate admissible groups. This and the results of the previous section will then be used to give a characterization of the admissible one-parameter groups over $\R^2$.\\

\begin{prop}\label{prop:diagadmiss}
Let $D$ be an $n\times n$ real diagonal matrix. Then $G_D$ is admissible if and only if $\tr(D)\neq0$.
\end{prop}\begin{proof}
If $D=0$, then $G_D$ is the trivial group which is obviously not admissible. So writing
$D=\begin{bmatrix}
d_1&&\\
&\ddots&\\
&&d_n
\end{bmatrix}$, we need only consider when some $d_l\neq 0$, and we assume without loss of generality that $d_n\neq 0$. Now consider the mapping $g(u_1,\ldots,u_{n-1},\tau)=e^{\tau D}\begin{bmatrix}u_1\\\vdots\\u_{n-1}\\1\end{bmatrix}
=\begin{bmatrix}u_1e^{\tau d_1}\\\vdots\\u_{n-1}e^{\tau d_{n-1}}\\e^{\tau d_n}\end{bmatrix}$ 
which is clearly a coordinate transform from $\R^n$ onto the ``upper half-space'' $\R^{n-1}\times\R^+$. For $i=1,\ldots,n-1$ we have $\dfrac{\partial g}{\partial u_i}=\begin{bmatrix}0\\\vdots\\e^{\tau d_i}\\\vdots\\0\end{bmatrix}$ where the unique nonzero entry in in the $i^\text{th}$ position, and
$\dfrac{\partial g}{\partial\tau}=\begin{bmatrix}d_1u_1e^{\tau d_1}\\\vdots\\d_{n-1}u_{n-1}e^{\tau d_{n-1}}\\d_ne^{\tau d_n}\end{bmatrix}$. Thus the Jacobian matrix of $g$ is 
$$\begin{bmatrix}
e^{\tau d_1}&0&\cdots&0&d_1u_1e^{\tau d_1}\\
0&e^{\tau d_2}&\cdots&0&d_2u_2e^{\tau d_2}\\
\vdots&\vdots&\ddots&\vdots&\vdots\\
0&0&\cdots&e^{\tau d_{n-1}}&d_{n-1}u_{n-1}e^{\tau d_{n-1}}\\
0&0&\cdots&0&d_ne^{\tau d_n}
\end{bmatrix}$$
and so the Jacobian determinant is 
$$J_g(u_1,\ldots,u_{n-1},\tau)=e^{\tau d_1}e^{\tau d_2}\cdots d_ne^{\tau d_n}=d_n e^{\tau(d_1+\cdots+d_n)}=d_ne^{\tau\tr(D)}.$$

Now suppose $\tr(D)=0$ and let $\psi$ be any function satisfying condition (\ref{admisscond}) for $G_{D}$ with Haar measure chosen as in the beginning of the proof of Theorem \ref{thm:mainthm}. Letting $u=\begin{bmatrix}u_1\\\vdots\\u_{n-1}\\1\end{bmatrix}$, we have that every orbit of the action of $G_D$ on the upper half-space contains a unique such $u$. Thus, the fact that $\Delta_\psi(\xi)=1$ for a.e.\ $\xi\in\R^n$ implies that $\Delta_\psi(u)=1$ for a.e.\ such $u$. So defining $du=du_1\cdots du_{n-1}$ we get
\begin{align*}
\|\psi\|_2^2&=\|\hat{\psi}\|_2^2\geq\dint_{\R^{n-1}\times\R^+}|\hat{\psi}(\xi)|^2d\xi\\
&=\dint_{\R^n}|\hat{\psi}\circ g(u_1,\ldots,u_{n-1},\tau)|^2|J_g(u_1,\ldots,u_{n-1},\tau)|\ du_1\cdots du_{n-1}d\tau\\
&=\dint_{\R^n}|\hat{\psi}(e^{\tau D}u)|^2|d_ne^{\tau\tr(D)}|\ du d\tau
=|d_n|\dint_{\R^{n-1}}\left(\dint_\R|\hat{\psi}(e^{\tau D}u)|^2\ d\tau\right)\ du\\
&=|d_n|\dint_{\R^{n-1}}\Delta_\psi(u)\ du=|d_n|\dint_{\R^{n-1}}du=\infty,
\end{align*}
so $\psi\notin L^2(\R^n)$. Thus, $G_D$ is not admissible.

Now say $\tr(D)\neq 0$. Since $G_D=G_{-D}$, we may assume $\tr(D)>0$. Let $u'=\begin{bmatrix}u_1\\\vdots\\u_{n-1}\\1\end{bmatrix}$, $u=\begin{bmatrix}u_1\\\vdots\\u_{n-1}\end{bmatrix}$, and let $|u|=\sqrt{u_1^2+\cdots+u_{n-1}^2}$, the euclidean norm of $u$ in $\R^{n-1}$. Define a region $R$ in $\R^n$ by $R=\{(u_1,\ldots,u_{n-1},\tau):\tau\in[-(|u|+1),-|u|]\}$, and let $\hat\psi$ be the characteristic function of 
$$S=\{\xi=(\xi_1,\ldots,\xi_n):(\xi_1,\ldots,\xi_{n-1},\pm\xi_n)\in g(R)\}.$$
Letting $V_{n-1}(r)$ denote the volume of the ball of radius $r$ in $\R^{n-1}$, we have
\begin{align*}
\|\hat{\psi}\|_2^2&=\dint_{S}d\xi=2\dint_{g(R)}d\xi
=2\dint_{R}|J_g(u_1,\ldots,u_{n-1},\tau)|\ du_1\cdots du_{n-1}d\tau\\
&=2\dint_{R}|d_ne^{\tau\tr(D)}|\ du d\tau
=2|d_n|\dint_{\R^{n-1}}\left(\dint_{-(|u|+1)}^{-|u|}e^{\tau\tr(D)}\ d\tau\right)\ du\\
&=2|d_n|\dint_{\R^{n-1}}\left[\frac{e^{-|u|\tr(D)}-e^{-(|u|+1)\tr(D)}}{\tr(D)}\right]\ du\\
&<\frac{2|d_n|}{\tr(D)}\dint_{\R^{n-1}}e^{-|u|\tr(D)}\ du
=\frac{2|d_n|}{\tr(D)}\dint_0^\infty V_{n-1}(r)e^{-r\tr(D)}\ dr\\
&=\frac{2|d_n|V_{n-1}(1)}{\tr(D)}\dint_0^\infty r^{n-1}e^{-r\tr(D)}\ dr\\
&=\frac{2|d_n|V_{n-1}(1)}{\tr(D)^{n+1}}\dint_0^\infty y^{n-1}e^{-y}\ dy
=\frac{2|d_n|V_{n-1}(1)\Gamma(n)}{\tr(D)^{n+1}}<\infty
\end{align*}
so $\hat{\psi}$ is in $L^2(\R^n)$, so it is actually the Fourier transform of some $\psi\in L^2(\R^n)$. We claim that $\psi$ is a wavelet for $G_D$.

By the definition of $S$, it is clear that it suffices to show that $\Delta_\psi(\xi)=1$ for \ $\xi\in\R^{n-1}\times\R^+$. Given such a $\xi$, let $\alpha=\ln(\xi_n)$, and let $v=e^{-\alpha D}\xi$. Then we have that $v$ is of the form $v=\begin{bmatrix}v_1\\\vdots\\v_{n-1}\\1\end{bmatrix}$, so
\begin{align*}
\Delta_\psi(\xi)&=\dint_\R |\hat{\psi}(e^{tD}\xi)|^2 dt
=\dint_\R |\hat{\psi}(e^{(t+\alpha)D}v)|^2 dt
=\dint_\R |\hat{\psi}(e^{tD}v)|^2 dt\\
&=\dint_\R |\hat{\psi}\circ g(v_1,\ldots,v_{n-1},t)|^2 dt=1
\end{align*}
because for fixed $u_1,\ldots,u_{n-1}$, the function $\hat\psi\circ g(u_1,\ldots,u_{n-1},\cdot)$ is the characteristic function of an interval of length 1. Thus $\psi$ is a wavelet for $G_D$, so $G_D$ is admissible.
\end{proof}

\begin{coro}\label{diagableadmiss}
Let $X$ be an $n\times n$ real diagonalizable matrix. Then $G_X$ is admissible if and only if $\tr(X)\neq0$.
\end{coro}\begin{proof}
Because $X$ is diagonalizable, we can write $X=CDC^{-1}$ for some $C\in GL_n(\R)$ and diagonal matrix $D$. So, by Lemma \ref{lemma:simadmiss}, $G_X$ is admissible if and only if $G_D$ is admissible. Proposition \ref{prop:diagadmiss} tells us that $G_D$ is admissible if and only if $\tr(D)\neq0$. Since $\tr(X)=\tr(CDC^{-1})=\tr(D)$, we see that $G_X$ is admissible if and only if $\tr(X)\neq0$.
\end{proof}

\begin{thm}
Let $X$ be a $2\times 2$ real matrix. Then $G_X$ is admissible if and only if $\tr(X)\neq0$.
\end{thm}\begin{proof}
We consider cases as to the diagonalizability of $X$.

Case 1: $X$ is diagonalizable (in $M_n(\R)$): This is just the 2-dimensional case of Corollary \ref{diagableadmiss}.

Case 2: $X$ is diagonalizable in $M_n(\C)$, but not in $M_n(\R)$: Let $\lambda=a+bi$ and $\bar\lambda=a-bi$ be the two eigenvalues of $X$. Since $\tr(X)=2a$, $\tr(X)\neq0$ if and only if $a\neq 0$. If $a\neq 0$, then by Corollary \ref{coro:complexdiagadmiss}, $G_X$ is admissible. So say $a=0$. Letting $B=\begin{bmatrix}1&1\\i&-i\end{bmatrix}$, we have $B^{-1}=\frac{1}{2}\begin{bmatrix}1&-i\\1&i\end{bmatrix}$ and we see that $X$ is similar to $B\begin{bmatrix}\lambda&0\\0&\bar\lambda\end{bmatrix}B^{-1}=\begin{bmatrix}0&b\\-b&0\end{bmatrix}$ in $M_n(\C)$, and so also in $M_n(\R)$. By Lemma \ref{lemma:simadmiss} we therefore only need to consider the case $X=\begin{bmatrix}0&b\\-b&0\end{bmatrix}$. Moreover, since $G_X=G_{\frac{1}{b}X}$ (by reparameterizing $t$ as $tb$), it suffices to consider $X=\begin{bmatrix}0&1\\-1&0\end{bmatrix}$.

Note that multiplying a vector $\begin{bmatrix}a\\b\end{bmatrix}$ by the matrix $X=\begin{bmatrix}0&1\\-1&0\end{bmatrix}$ gives the same result as viewing the vector as the complex number $a+bi$ and multiplying by $-i$ (note that if we use $X^T$, this changes to $i$). So we are really considering the group $\{e^{i\theta}:\theta\in\R\}=\{e^{i\theta}:\theta\in[0,2\pi)\}=SO(2)$ (this can also be verified by computing the infinite series, using the fact that $X^4=\begin{bmatrix}1&0\\0&1\end{bmatrix}$). Now if $\psi$ is a function satisfying condition (\ref{admisscond}), and we represent a $\xi\neq0$ as $\xi=\rho e^{i\ph}$, then we have
$$1=\Delta_\psi(\xi)=\dint_0^{2\pi}\left|\hat{\psi}\left(e^{i\theta}\rho e^{i\ph}\right)\right|^2 d\theta
=\dint_0^{2\pi}\left|\hat{\psi}\left(\rho e^{i(\theta+\ph)}\right)\right|^2 d\theta
=\dint_0^{2\pi}\left|\hat{\psi}\left(\rho e^{i\theta}\right)\right|^2 d\theta $$
because $e^{i\theta}$ is $2\pi$ periodic and $d\theta$ is a Haar measure for $\R$, and so
$$\|\psi\|_2^2=\|\hat{\psi}\|_{2}^2=\dint_{\R^2-\{0\}}\left|\hat{\psi}(\xi)\right|^2d\xi
=\int_0^\infty \rho \dint_0^{2\pi} \left|\hat{\psi}\left(\rho e^{i\theta}\right)\right|^2 d\theta d\rho
=\int_0^\infty \rho d\rho
=\infty,$$
so $\psi\notin L^2(\R^2)$. Thus, $G_X$ is not admissible.

Case 3: $X$ is not diagonalizable in $M_n(\C)$: In this case, $X$ is similar to its Jordan normal form, $\begin{bmatrix}\lambda&1\\0&\lambda\end{bmatrix}$. So $\tr(X)=2\lambda$, and $\tr(X)\neq0$ iff $\lambda\neq0$. If $\lambda\neq 0$, then $X$ is similar to $Y=
\begin{bmatrix}1&0\\0&\frac{1}{\lambda}\end{bmatrix}
\begin{bmatrix}\lambda&1\\0&\lambda\end{bmatrix}
\begin{bmatrix}1&0\\0&\lambda\end{bmatrix}=\begin{bmatrix}\lambda&\lambda\\0&\lambda\end{bmatrix}$, so Lemma \ref{lemma:simadmiss} gives us that $G_X$ is admissible if $G_Y$ is. Now the symmetric part of $Y$ is $M_Y=\begin{bmatrix}\lambda&\lambda/2\\\lambda/2&\lambda\end{bmatrix}$, which has determinant $\det(M_Y)=\lambda^2-\frac{\lambda^2}{4}=\frac{3\lambda^2}{4}$. Since this determinant is positive, it follows that the eigenvalues of $M_Y$, the product of which is the determinant, have the same sign. So by Corollary \ref{coro:symmpart}, $G_Y$ is admissible, so $G_X$ is admissible.

If $\lambda=0$, then $X$ is similar to $\begin{bmatrix}0&1\\0&0\end{bmatrix}$, so by Lemma \ref{lemma:simadmiss}, it suffices to suppose $X=\begin{bmatrix}0&1\\0&0\end{bmatrix}$. Let $Y=X^T$, and note that because $Y^2=0$, $e^{tY}=\begin{bmatrix}1&0\\t&1\end{bmatrix}$. Define the mapping $g:(\R-\{0\})\times\R\to(\R-\{0\})\times\R$ by $g(u,t)=e^{tY}\begin{bmatrix}u\\0\end{bmatrix}=\begin{bmatrix}u\\tu\end{bmatrix}$. We have that $\dfrac{\partial g}{\partial u}=\begin{bmatrix}1\\t\end{bmatrix}$ and $\dfrac{\partial g}{\partial t}=\begin{bmatrix}0\\1\end{bmatrix}$, so the Jacobian matrix of $g$ is $\begin{bmatrix}1&0\\t&1\end{bmatrix}$, and the Jacobian determinant of $g$ is $J_g(u,t)=1$. Note that the image of $g$ is almost all of $\R^2$, and this means that for almost all $\xi\in\R^2$, there is a unique $u$ such that $\begin{bmatrix}u\\0\end{bmatrix}$ and $\xi$ are in the same orbit of the action of $\{e^{tY}:t\in\R\}$ on $\R^2$. Further, for a fixed $\xi$ in the image, the corresponding $u$ is $u=\xi_1$, and so the mapping $\xi\to u$ preserves sets of measure 0. 
Thus, if $\psi $ is a function satisfying condition (\ref{admisscond}), then since the integral is over orbits of $\{e^{tY}:t\in\R\}$, $\Delta_\psi\left(\begin{bmatrix}u\\0\end{bmatrix}\right)=1$ for a.e.\ $u$. For such a $\psi$,
\begin{align*}
\|\psi\|_2^2&=\|\hat{\psi}\|_2^2=\dint_{(\R-\{0\})\times\R}\!\!\!\!\!\!\!\!\!\!\!\!\!|\hat{\psi}(\xi)|^2\ d\xi
=\dint_{(\R-\{0\})\times\R}\!\!\!\!\!\!\!\!\!\!\!\!\!|\hat{\psi}\circ g(u,t)|^2|J_g(u,t)|\ du dt\\
&=\dint_{\R-\{0\}}\dint_\R\left|\hat{\psi}\left(e^{tY}\begin{bmatrix}u\\0\end{bmatrix}\right)\right|^2\ dt\ du
=\dint_{\R-\{0\}}\Delta_\psi\left(\begin{bmatrix}u\\0\end{bmatrix}\right)\ du\\
&=\dint_{\R-\{0\}}du=\infty,
\end{align*}
so $\psi\notin L^2(\R^2)$, showing that $G_X$ is not admissible.
\end{proof}

\bibliographystyle{plain}
\bibliography{oneparameter}

\begin{thebibliography}{1}

\bibitem{haar2007}
Ilya~A. Krishtal, Benjamin~D. Robinson, Guido~L. Weiss, and Edward~N. Wilson.
\newblock Some simple {H}aar-type wavelets in higher dimensions.
\newblock {\em J. Geom. Anal.}, 17(1):87--96, 2007.

\bibitem{LWW2013}
Demetrio Labate, Guido Weiss, and Edward Wilson.
\newblock Wavelets.
\newblock {\em Notices Amer. Math. Soc.}, 60(1):66--76, 2013.

\bibitem{LWWW2002}
R.~S. Laugesen, N.~Weaver, G.~L. Weiss, and E.~N. Wilson.
\newblock A characterization of the higher dimensional groups associated with
  continuous wavelets.
\newblock {\em J. Geom. Anal.}, 12(1):89--102, 2002.

\bibitem{weiss}
G.~Weiss and E.~N. Wilson.
\newblock The mathematical theory of wavelets.
\newblock In {\em Twentieth century harmonic analysis---a celebration ({I}l
  {C}iocco, 2000)}, volume~33 of {\em NATO Sci. Ser. II Math. Phys. Chem.},
  pages 329--366. Kluwer Acad. Publ., Dordrecht, 2001.

\end{thebibliography}
\end{document}